\documentclass[reqno, oneside]{amsart}
\usepackage[utf8]{inputenc}    
\usepackage[backend=biber, style=alphabetic]{biblatex}
\usepackage[shortlabels,inline]{enumitem} % Enhances the enumerate-environment.
\usepackage[english]{babel}       % For proper word breaking at line ends.
\usepackage{amsthm}               % Amsthm is part of the (required) AMS-LaTeX (amslatex) distribution, so should be present in any LaTeX distribution. 
\usepackage{amsmath,amssymb}      % Principal package in the AMS-LaTeX, see http://tug.ctan.org/cgi-bin/ctanPackageInformation.py?id=amsmath
\usepackage{amsfonts}             
\usepackage{dsfont}              % Doublestroke (natural numbers, rational numbers, etc.)
\usepackage{stmaryrd}           % Part of amsfonts but needs to be loaded for additional commands.
\usepackage{mathrsfs}
\usepackage{manfnt}

\usepackage[dvipsnames]{xcolor}
\usepackage[a4paper, left=1.5in, right=1.5in, top=1.5in, bottom=1.5in]{geometry} % For page size and margin settings.
\usepackage{graphicx}           
\usepackage{mathtools}          % For greek math symbol formatting.
\usepackage{enumitem}           % For control of 'enumerate' numbering.
\usepackage{listings}           % For control of 'itemize' spacing.
\usepackage{todonotes}          % For clear TODO notes.
\usepackage{tikz}\usetikzlibrary{matrix,arrows,calc,decorations.pathmorphing} % For diagrams and drawings.
\usepackage{tikz-cd}            % For commutative diagrams.
\usepackage[autostyle]{csquotes} % For nice quote styles.
\usepackage{hyperref}           % Page numbers and '\ref's become clickable.
\usepackage{adjustbox}
\usepackage{comment}

\newtheorem{FactCounter}{dummy}[section] % Defines the counter.
\newtheorem{theorem}[FactCounter]{Theorem}
\newtheorem{corollary}[FactCounter]{Corollary}
\newtheorem{lemma}[FactCounter]{Lemma}
\newtheorem{proposition}[FactCounter]{Proposition}

\theoremstyle{definition}
\newtheorem{definition}[FactCounter]{Definition}

\theoremstyle{definition}

\theoremstyle{remark}
\newtheorem{remark}[FactCounter]{Remark}

\theoremstyle{remark}

\theoremstyle{remark}

\theoremstyle{remark}

\theoremstyle{definition}
\newtheorem{example}[FactCounter]{Example}

\theoremstyle{definition}

\theoremstyle{plain}
\newtheorem{fact}[FactCounter]{Fact}

\newcommand{\onto}{\twoheadrightarrow}

\newcommand{\lra}{\longrightarrow}%
\newcommand{\ra}{\rightarrow}%
\newcommand{\Ra}{\Rightarrow}%

\DeclareRobustCommand\lonto
     {\relbar\joinrel\twoheadrightarrow}
\DeclareRobustCommand\linto
     {\lhook\joinrel\longrightarrow}

\let\phi\varphi

\calclayout

\calclayout

\hypersetup{colorlinks=true}    % Gives color to hyperlinks

\colorlet{red}{BrickRed}
\colorlet{green}{OliveGreen}
\colorlet{magenta}{Thistle}

\makeatletter
\def\@tocline#1#2#3#4#5#6#7{\relax
	\ifnum #1>\c@tocdepth % then omit
	\else
	\par \addpenalty\@secpenalty\addvspace{#2}%
	\begingroup \hyphenpenalty\@M
	\@ifempty{#4}{%
		\@tempdima\csname r@tocindent\number#1\endcsname\relax
	}{%
		\@tempdima#4\relax
	}%
	\parindent\z@ \leftskip#3\relax \advance\leftskip\@tempdima\relax
	\rightskip\@pnumwidth plus4em \parfillskip-\@pnumwidth
	#5\leavevmode\hskip-\@tempdima
	\ifcase #1
	\or\or \hskip 1em \or \hskip 2em \else \hskip 3em \fi%
	#6\nobreak\relax
	\dotfill\hbox to\@pnumwidth{\@tocpagenum{#7}}\par
	\nobreak
	\endgroup
	\fi}
\makeatother

\begin{document}

\date{\today}
\subjclass[2020]{06F20 (Primary), 03C64, 12J15}
\keywords{cross-section, totally ordered abelian group, Dedekind complete, tame pair, real closed valued field}
\title{\texorpdfstring{Cross-Sections of Divisible Abelian $ \MakeLowercase{O} $-Groups via Tame Pairs}{Sections of Divisible Abelian  o-Groups via Tame Pairs}}
\author{Ricardo  Palomino Piepenborn}
\email{\href{mailto:ricardo.palomino@rjpp.net}{ricardo.palomino@rjpp.net}}
\urladdr{\href{http://www.rjpp.net}{rjpp.net}}

\begin{abstract}
	It is shown that images of cross-sections of surjective morphisms $ f : \Gamma \lonto \Delta $ of divisible abelian $ o $-groups are exactly divisible, tame (equivalently, relative Dedekind complete) and cofinal subgroups of $ \Gamma  $ compatible with $ f $ in a suitable sense. The note concludes with an application  to real closed valued fields.
\end{abstract}

\maketitle

\section{Introduction}

Recall that if $ f: A \lra B $ and $ g: B \lra A $ are morphisms in a given category $ \mathcal{C} $ such that $ f \circ g = \text{id}_B $, then $ g $ is a \textit{cross-section of $ f $} and $ f $ is a \textit{retract of $ g $}. Cross-sections are monomorphisms and retracts are epimorphisms. In particular, if $ \mathcal{C} $ is either the category of abelian groups or the category of commutative unital rings, and $ f: A \lonto B $ is a surjective morphism with a cross-section $ g : B \linto A $, then  $ g(B) \subseteq A$ is a substructure of $ A $ which is isomorphic to $ B $. Conversely, if $ C \subseteq A$ is any substructure such that the restriction $ f_{\upharpoonright C} : C \lra B $ is an isomorphism, then $ (f_{\upharpoonright C})^{-1}: B  \linto A $ is a cross-section of $ f $. In this way, a cross-section of $ f $ allows to access $ B $  \enquote {internally} within $ A $. This use of cross-sections is deployed in the model theory of valued fields, where the three-sorted valued field structure  is expanded by a cross-section for the valuation and a cross-section for the residue field map when they exist, see for example \cite{vdD.model_theory_val_field}.

Let now $ K $ be a real closed field and $ V \subseteq K $ be a convex subring, that is,  $ (K, V) $ is a \textit{real closed valued field}. It is well-known that  if a subfield $ F \subseteq V $ is maximal in $ V $ with respect to subset inclusion, then the restriction of the residue field map $ \lambda: V \lonto V/\mathfrak{m}_V $ to $ F $ is an isomorphism (\cite[Proposition 2.5.3]{knebusch/scheiderer.real_algebra}), therefore maximal subfields of $ V $ are images of cross-sections of $ \lambda $, and the converse also holds. The work of Lewenberg and van den Dries in \cite{vdD/Lewenberg.T-convexity_and_tame_extensions} shows that there is another characterization of images of cross-sections of $ \lambda $: these are exactly real closed subfields of $ V $ which are \textit{tame} (or \textit{Dedekind complete}) \textit{in $ K $}, see Definition \ref{def.tame}.  Conversely, if $ L \subseteq K $ is a real closed subfield which is tame in  $ K $, then the convex hull of $ L  $ in $ K $ is a convex subring of  $ K $ such that the restriction of its residue field map to $ L $ is an isomorphism. In summary,  the order-theoretic notion of tameness  characterizes cross-sections of the residue field map of a real closed valued field.

This  note shows that tameness also characterizes cross-sections of surjective morphisms of divisible $ o $-groups (divisible totally ordered abelian groups), see Theorem \ref{the.equiv_sect_DOAGS}. The proof is very elementary. Some parts of the proof are simplified using a result on tame pairs of o-minimal structures (see the end of Section \ref{sec.tame}), but a purely order-theoretical proof is also given without any reference to o-minimality. The discussion after Theorem \ref{the.equiv_sect_DOAGS} recalls the known fact that these cross-sections always exist. Theorem \ref{the.equiv_sect_DOAGS} is applied in Section \ref{sec.rcvf} to the valuation map of a real closed valued field. This yields a one-sorted axiomatization of real closed valued fields by regarding them as pairs $ (K; G) $, where $ K $ is a real closed field and $ G \subseteq K^{>0} $ is a divisible multiplicative subgroup tame in $ K^{>0} $ and which satisfies some first-order conditions.

 Throughout this note all groups are abelian, and an $ o $\textit{-group} is a totally ordered (abelian) group. All groups are written additively with the exception of Section \ref{sec.rcvf}.

\section{Preliminaries on tame pairs}\label{sec.tame}
In this section $ A=(A, < ,\ldots) $ and $ B=(B, <, \ldots) $ are expansions of dense linear orders without endpoints in the same first-order language $ \mathscr{L} $ and  $ A \subseteq B $ as $ \mathscr{L} $-structures. For example, $ \mathscr{L} $ could be the language $ \mathscr{L}^{\text{og}}:= \{ +, - , 0 , < \}  $ of $ o $-groups and $ A \subseteq B $ be an extension of divisible $ o  $-groups (recall that the underlying order of a divisible $ o $-group is a dense linear order without endpoints). 

\begin{definition}[1.12 in \cite{vdD/Lewenberg.T-convexity_and_tame_extensions}, or \cite{pillay.definability_of_types}]\label{def.tame}
 $ A $ \textit{is tame in $ B $} (or $ A $ \textit{is Dedekind complete in $ B $}) if for every $ A $-bounded $ b \in B $ (that is, for every $ b \in \text{c.h.}_B(A) := \{b' \in B \mid \exists a_1, a_2 \in A \text{ such that } a_1 \leq b'\leq  a_2\} $) there exists $ a \in A $  such that one of the following items holds true:
	\begin{enumerate}[\normalfont(i)]
	\item $b =a  $,	or
	\item $ b <a $ and there is no $ a' \in A $ such that $ b < a' < a $, or
	\item  $ a < b $ and there is no  $ a'  \in A$  such that $ a < a' < b$.
	\end{enumerate}
\end{definition}

Other well-known equivalent characterizations of tameness are collected in the next statement. The proof is straightforward from the definitions. 

\begin{fact}
	The following are equivalent:
		\begin{enumerate}[\normalfont(i)]
		\item $ A  $ is tame in $ B $.
		 \item For every $ b \in B $, the set  $ \{ a \in A \mid a < b \}  $ has a supremum in  $ A \cup \{ \pm \infty \}  $.
		\end{enumerate}
		Suppose further that $ A $ and $ B $ are  expansions of divisible $ o $-groups.  Then \emph{(i)} and  \emph{(ii)} are equivalent to:
		\begin{enumerate}
			\item[\normalfont(iii)] 	For all $ b \in \emph{c.h.}_B(A) $ there exists $ a \in A $ such that  $ |b-a|< a' $ for all  $ a' \in A^{>0}  $.
		\end{enumerate}
\end{fact}

\begin{remark}\label{rem.tame}
	 If $ A  $ is tame in $ B $, then it follows from the fact that $ (A, <)$ is a dense linear order that for every  $ A $-bounded $ b \in B $ (that is, for every $ b \in \text{c.h.}_B(A) $) there exists a unique  $ a \in A $ such that exactly one of the items  (i) - (iii) in Definition \ref{def.tame} holds for $ a $ and $ b $.
\end{remark}

\begin{definition}\label{def.st_part_map}
 The \textit{standard part map associated with the tame pair $ A \subseteq B $} is the map $ \text{st}^B_A: \text{c.h.}_B(A) \lonto A$ given by setting  $ \text{st}^B_A(b)  $  ($ b \in \text{c.h.}_B(A) $) to be the unique element in  $ A $ for which one of the items  (i) - (iii) in Definition \ref{def.tame} hold for $ \text{st}^B_A(b) $ and $ b $. If $ A $ and $ B $ are clear from the context, then write $ \text{st}:= \text{st}^B_A $.	
\end{definition}

\begin{remark}\label{rem.st_part_map}
If $ A $ is tame in $ B $, then:
\begin{enumerate}[\normalfont(i), ref=\ref{rem.st_part_map} (\roman*)]
			\item\label{rem.st_part_map.i} $ \text{st}(a) =a $ for all $ a \in A \subseteq \text{c.h.}_B(A) $.
			\item\label{rem.st_part_map.ii} If $ b \in \text{c.h.}_B(A) $  and $  b <  \text{st}(b) $, then it follows from the fact that $ (A, < ) $ is a dense linear order  that for all $ a \in A $ such that  $ a < b $ there exists $  a' \in A $ such that  $ a < a' < b $. 
\end{enumerate}
\end{remark}

\begin{lemma}\label{lem.st_ord_pres}
Suppose that $ A $ is tame in $ B $. 	The standard part map $ \emph{st}: \emph{c.h.}_B(A) \lonto A $ is order-preserving. 
\end{lemma}

\begin{proof}
	Assume for contradiction that $ b_1, b_2 \in \text{c.h.}_B(A) $ are such that $ b_1 < b_2 $ and  $ \text{st}(b_2) < \text{st}(b_1) $. Note that $ \text{st}(b_1) \neq b_2 $ and $ \text{st}(b_2)  \neq b_1$, as otherwise $ \text{st}(b_1) = \text{st}(b_2) $ by Remark \ref{rem.st_part_map.i}.

\noindent \underline{Case 1:}   $ \text{st}(b_1) \leq b_1 $. In this case $ \text{st}(b_2) < \text{st}(b_1) < b_2 $, therefore $ \text{st}(b_1)\in A$ contradicts tameness. 

\noindent \underline{Case 2:}  $  b_1 < \text{st}(b_1)$.  Tameness together with $ b_1 < \text{st}(b_1) $ and $ \text{st}(b_2)< \text{st}(b_1) $ imply  $ \text{st}(b_2) <   b_1 < \text{st}(b_1) $, and  $ b_1 < b_2 $ then implies that either $ \text{st}(b_2) <  b_1 < b_2 < \text{st}(b_1) $ or  $  \text{st}(b_2) <   b_1 <  \text{st}(b_1) <  b_2 $. In the latter case $  \text{st}(b_2) <    \text{st}(b_1) <  b_2 $ and $ \text{st}(b_1) \in A $ contradict tameness. In the former case, Remark \ref{rem.st_part_map.ii} implies that there exists $ a \in A $ such that  $ \text{st}(b_2) < a < b_1 $,  therefore    $ \text{st}(b_2) < a < b_2 $  and $ a \in A $ contradict tameness.
\end{proof}

If $ (A, <, \dots) \preceq (B, < , \dots) $ is an elementary extension of  o-minimal structures (\cite{tame}) and $ D\subseteq A^n $ is an $ A $-definable subset of $ A $, then write $ D_B $ for the definable subset in $ B^n $ given by the same formula defining  $ D $ in $ A^n $. Moreover, say that $ \overline{b}= (b_1, \dots, b_{n}) \in B^n $ is \textit{$ A $-bounded}  if $ b_i $ is $ A $-bounded for all $ i \in \{1,\dots, n\} $, and if $ A  $ is tame in $ B $ and $ \overline{b} \in B^n$ is $ A $-bounded, write $ \text{st}(\overline{b}) $ for $ (\text{st}(b_1), \dots, \text{st}(b_{n}))  $. 

\begin{lemma}\label{lem.st_cont_def_comm}	
	Let $ (A, < , \dots) $ be $ o $-minimal and tame in an elementary extension $ (B, < , \dots) $. Let $ f: D \lra A $ be a continuous $ A $-definable function on an $ A $-definable set $ D\subseteq A^n $, and let $ \overline{b} \in D_B $ be  $ A $-bounded with $ \emph{st}(\overline{b})\in D $. Then $ f_B(\overline{b}) $ is $ A $-bounded and $  	\emph{st}(f_B(\overline{b})) = f(\emph{st}(\overline{b})) $.	
\end{lemma}

\begin{proof}
See 1.13 in \cite{vdD/Lewenberg.T-convexity_and_tame_extensions}.
\end{proof}

\begin{proposition}\label{prop.retract_tame_subgrp}
	Let  $ \Delta \subseteq \Gamma $ be an extension of divisible $ o $-groups. Suppose that $ \Delta $ is tame and cofinal in $ \Gamma $. Then $ \emph{st}: \Gamma  \lonto \Delta$ is a  surjective $ o $-group homomorphism. In particular, $ \emph{st}: \Gamma \lonto \Delta $ is a retract of the $ o $-group extension $ \Delta \subseteq \Gamma $.
\end{proposition}

\begin{proof}
	First note that $ \text{c.h.}_{\Gamma}(\Delta) = \Gamma $ since $ \Delta $ is cofinal in $ \Gamma $, therefore $ \text{st} $ is a map $ \Gamma \lonto \Delta $. The map $ \text{st}: \Gamma \lonto \Delta $ is order-preserving by Lemma \ref{lem.st_ord_pres}.  Let $ \mathscr{L}^{\text{og}}:= \{+, -, 0, \leq\} $ be the language of $ o $-groups. Since the $ \mathscr{L}^{\text{og}} $-theory of divisible $ o  $-groups is model complete and o-minimal (\cite[Corollary 3.1.17]{marker_model_2002}), $ \Delta \subseteq \Gamma $ is an elementary extension of $ o $-minimal $ \mathscr{L}^{\text{og}} $-structures. Moreover,  $ \Delta $ and  $ \Gamma $ are topological groups with respect to the order topology (i.e., $ + $ and  $ - $ are continuous functions), therefore it follows from Lemma \ref{lem.st_cont_def_comm} that $ \text{st}(\gamma_1+ \gamma_2) = \text{st}(\gamma_1) + \text{st}(\gamma_2)  $ for all $  \gamma_1, \gamma_2 \in \Gamma $, and thus $ \text{st}: \Gamma \lonto \Delta $ is a surjective $ o $-group homomorphism such that $ \text{st}_{\upharpoonright \Delta} = \text{id}_{\Delta} $, so it is a retract of $ \Delta \subseteq \Gamma $. 
\end{proof}

\section{Main Theorem}

\begin{theorem}\label{the.equiv_sect_DOAGS}
	Let $ f : \Gamma \lonto \Delta $ be a surjective $ o $-group homomorphism of divisible $ o $-groups and $ \Delta' \subseteq \Gamma $ be a subgroup. The following are equivalent:
	\begin{enumerate}[\normalfont(i)]
		\item The map $ f_{\upharpoonright \Delta'} : \Delta' \lra \Delta $ is an $ o $-group isomorphism. In particular,   $ (f_{\upharpoonright \Delta'})^{-1}: \Delta \linto \Gamma$ is a cross-section of the $ o $-group homomorphism $f : \Gamma \lonto \Delta$.
		\item  $ \Delta'$ is divisible, tame and cofinal in $ \Gamma $, and $ \emph{ker}(f)= \{ \gamma \in \Gamma\mid \emph{st}(\gamma)=0 \}  $.
		\item $ \Delta'$ is divisible, tame and cofinal in $ \Gamma $, and $ f(\gamma)\geq 0$ if and only if $ \emph{st}(\gamma) \geq 0 $ for all $ \gamma \in \Gamma $.
	\end{enumerate}
	In particular, if any of the items \emph{(i)} - \emph{(iii)} hold, then:
\begin{enumerate}[\normalfont-]
\item 	$ \Delta' $ is divisible, tame and cofinal in $ \Gamma $,
\item $ \emph{st}(\gamma) $ is the unique element in  $ \Delta' $ such that  $ f(\emph{st}(\gamma)) =f(\gamma) $ for all $ \gamma \in \Gamma $, and
\item the standard part map $ \emph{st}: \Gamma \lonto \Delta' $ is a retract of the $ o $-group extension $ \Delta' \subseteq \Gamma $.
\end{enumerate}
\end{theorem}

\begin{proof}
	
	\underline{(i) $ \Ra $  (ii).} Clearly $ \Delta' $ is divisible, therefore $ (\Delta', <) $ is a dense linear order without endpoints. To prove that $ \Delta' $ is cofinal in $ \Gamma $, pick $ \gamma \in \Gamma $ with $ 0 < \gamma $. Since $ (\Delta, <) $ has no end points, there exists $ \delta \in \Delta $ with $ f(\gamma) < \delta $, and since $ f_{\upharpoonright \Delta'} $ is surjective, there exists $ \delta' \in \Delta'$ such that $ f(\delta') = \delta $. But then $ \gamma < \delta' $, as otherwise  $ \delta' \leq \gamma $ would imply that  $ \delta =f(\delta') \leq f(\gamma) $, hence  $ \Delta' $ is cofinal in $ \Gamma $ and thus every $ \gamma \in \Gamma $ is $ \Delta' $-bounded. To prove that $ \Delta'$ is tame in $\Gamma $, pick any $ \gamma \in \Gamma $ and assume without loss of generality that $ 0 < \gamma $  (otherwise replace $ \gamma  $ by $ -\gamma $). Since $ f_{\upharpoonright \Delta'} $ is bijective by assumption,  there exists a unique $  \delta' \in \Delta'$ such that $ f(\gamma) = f(\delta') $, i.e., $ \delta'- \gamma\in \text{ker}(f)  $. Note that $0 \leq \delta' $, as otherwise $ \delta' < 0 $ implies that  $ f(\gamma) =f(\delta') < f(0) $  since $ f $ is order-preserving and $ f_{\upharpoonright \Delta'} $ is injective, and $ 0 < \gamma $ implies $ f(0) \leq  f(\gamma) $, giving the required contradiction. It is now claimed that $ \text{st}(\gamma) = \delta' $. If $ \gamma \in \Delta' $, then  $ \gamma = \delta' $ by choice of  $ \delta' \in \Delta' $ and thus  $ \text{st}(\gamma)= \text{st}(\delta') = \delta' $. If $ \gamma \notin \Delta' $, then there are two possible cases:

	\begin{enumerate}[-]
		\item Case 1:  $  \gamma <  \delta' $. Assume  for contradiction that there exists $ \delta'_1 \in \Delta'$  such that $    \gamma < \delta_1'< \delta'  $.  Then $ 0 < \delta_1' - \gamma < \delta' - \gamma $, and since  $\text{ker}(f)$ is convex in $ \Gamma $ and $ \delta' - \gamma \in \text{ker}(f) $, it follows that $ \delta_1'- \gamma \in  \text{ker}(f)$, hence $ f(\delta_1') = f(\gamma) = f(\delta') $, contradicting  uniqueness of $ \delta' \in \Delta'$.
		\item Case 2: $  \delta' < \gamma $. Analogous to Case 1.
	\end{enumerate}

\noindent Therefore $ \Delta' $ is tame in $\Gamma $. In particular, this shows that for every $ \gamma \in \Gamma $,  $ \text{st}(\gamma)  $ is the unique element in $ \Delta'$ such that  $ f(\gamma) =  f(\text{st}(\gamma)  )  $, i.e., $ \text{st}(\gamma)  $ is the unique element in $ \Delta' $ such that $ \eta_{\gamma}: =\gamma - \text{st}(\gamma) \in \text{ker}(f) $, hence \[
	f(\gamma) = 0 \iff f(\text{st}(\gamma)+\eta_{\gamma}) = 0 \iff f(\text{st}(\gamma)) = 0 \iff \text{st}(\gamma) = 0
,\] 
 where the last equivalence follows from the assumption that $ f_{\upharpoonright \Delta'}  $ is injective. 

 \underline{(ii) $ \Ra $  (i).}  $ \text{ker}(f)=\{ \gamma\in\Gamma\mid\text{st}(\gamma)=0 \}   $ implies  that $ \Delta' \cap \text{ker}(f) = (0)$, and thus $ f_{\upharpoonright \Delta'} $ is injective.  To show that 	$ f_{\upharpoonright \Delta'}$ is surjective it suffices to prove that $ f(\gamma) = f(\text{st}(\gamma)) $ for all $ \gamma \in \Gamma$ (note that since $ \Delta' $ is cofinal in  $ \Gamma $,  $ \text{st}(\gamma) $ exists for all $ \gamma \in \Gamma $). Since $ \text{ker}(f)=\{ \gamma\in\Gamma\mid\text{st}(\gamma)=0 \}   $, it suffices in turn to show that $ \text{st}(\gamma-\text{st}(\gamma))= 0 $ for all $ \gamma \in \Gamma $. But this is indeed the case, since \[
	 \text{st}(\gamma-\text{st}(\gamma))= 0 \overset{(*)}{\iff} \text{st}(\gamma) - \text{st}(\text{st}(\gamma)) = 0  \iff \text{st}(\gamma) = \text{st}(\text{st}(\gamma)) = \text{st}(\gamma),
 \] where $ (*) $ follows by Proposition \ref{prop.retract_tame_subgrp}. An alternative proof without Proposition \ref{prop.retract_tame_subgrp} is as follows. If $ \gamma \in \Delta' $, then $ \gamma = \text{st}(\gamma) $ and thus  $ f(\gamma) = f(\text{st}(\gamma)) $. Let now $ \gamma \in \Gamma \setminus \Delta' $, assume without loss of generality that $ 0  < \gamma$ (otherwise replace $ \gamma $ by $- \gamma $), and assume for contradiction that $ \text{st}(\gamma- \text{st}(\gamma))\neq 0 $. Note that $ \gamma \notin \Delta' $ implies  $ \gamma - \text{st}(\gamma) \neq \text{st}(\gamma - \text{st}(\gamma))$; moreover $ \text{st}(\gamma) \neq 0 $, as otherwise $ \text{st}(\gamma - \text{st}(\gamma))=0 $, therefore $ 0 < \text{st}(\gamma) $ by Lemma \ref{lem.st_ord_pres}.
\begin{enumerate}[\normalfont-]
	\item 	Case 1: $ 0 < \gamma < \text{st}(\gamma) $. Then $  \gamma -\text{st}(\gamma)< 0 $, and there are  2 possible subcases:
	\begin{enumerate}[\normalfont-]
		\item Subcase 1.1: $ \text{st}(\gamma - \text{st}(\gamma)) < \gamma- \text{st}(\gamma) < 0 $.  In this case, by Remark \ref{rem.st_part_map.ii} there exists $ \delta' \in \Delta' $ such that  $   \gamma - \text{st}(\gamma) < \delta' < 0$, hence $ \gamma  < \delta+ \text{st}(\gamma) < \text{st}(\gamma) $ and $ \delta'+ \text{st}(\gamma) \in \Delta' $ is a contradiction to tameness of $ \Delta' $ in $ \Gamma $.
		\item Subcase 1.2: $  \gamma- \text{st}(\gamma) < \text{st}(\gamma - \text{st}(\gamma)) < 0 $. In this case, $ \gamma  < \text{st}(\gamma) + \text{st}(\gamma - \text{st}(\gamma))  < \text{st}(\gamma)$ and $ \text{st}(\gamma) +\text{st}(\gamma - \text{st}(\gamma))  \in \Delta'$ is  a contradiction to tameness of $ \Delta' $ in $ \Gamma $.
	\end{enumerate}
\item Case 2: $ 0 < \text{st}(\gamma) < \gamma $. Analogous to Case 1.
\end{enumerate}
In each of the cases above a contradiction is reached, hence $  \text{st}(\gamma - \text{st}(\gamma)) =0  $ for all $ \gamma \in \Gamma $, i.e., $f(\gamma) = f(\text{st}(\gamma)) $ for all $ \gamma \in \Gamma $,  and thus $ f_{\upharpoonright \Delta'}: \Delta' \lra \Gamma $ is surjective, as required.

\underline{(ii) $ \Leftrightarrow $ (iii).}  The implication (iii) $ \Ra$ (ii) is clear, so suppose that item (ii) holds, i.e., $ f(\gamma) =  0$ if and only if $ \text{st}(\gamma) = 0 $ for all $ \gamma \in \Gamma $. It therefore suffices to show that $ f(\gamma) >0 $ if and only if $ \text{st}(\gamma) >0$ for all $ \gamma \in \Gamma$. Pick $ \gamma \in \Gamma $. It will be shown that $ f(\gamma)>0 $ implies $ \text{st}(\gamma)>0 $; the proof of the other direction is similar. Assume for contradiction that $ f(\gamma) >0 $ and $ \text{st}(\gamma) \leq 0 $. Since $ \text{st}(\gamma) =0  $ implies $  f(\gamma)=0$, it must be the case that $ \text{st}(\gamma) < 0 $, and thus $ \gamma \leq 0 $, as otherwise  $  \text{st}(\gamma) < 0  < \gamma$ contradicts tameness of $ \Delta' $ in $ \Gamma $. On the other hand, $ 0< f(\gamma) $ implies that $ 0 \leq \gamma $, as otherwise $  \gamma < 0$ implies  $ f(\gamma) \leq 0$,  therefore $ 0 = \gamma $ and thus  $ f(0) = f(\gamma) >0$, a contradiction.

The last statement of the theorem is clear by Proposition \ref{prop.retract_tame_subgrp}. Alternatively, suppose that any of the items (i) - (iii) hold, so that $ \Delta' $ is tame and cofinal in $ \Gamma $, and $ f_{\upharpoonright \Delta'}: \Delta' \lra \Delta $ is an $ o $-group isomorphism. Then it follows from the proof of the implication (i) $ \Ra $  (ii) that  $ \text{st}(\gamma)  $ is the unique element in $ \Delta' $ such that $ f(\text{st}(\gamma)) = f(\gamma) $  for all $ \gamma \in \Gamma $, hence $\text{st}(\gamma) = (f_{\upharpoonright \Delta'})^{-1}(f(\gamma))$ for all $ \gamma \in \Gamma $ and thus  $ \text{st} =  (f_{\upharpoonright \Delta'})^{-1} \circ f$ is a surjective $ o $-group homomorphism such that $ \text{st}_{\upharpoonright \Delta'} = \text{id}_{\Delta'} $, therefore $ \text{st} : \Gamma \lra \Delta' $ is a retract of the $ o $-group extension $ \Delta' \subseteq \Gamma $. 
\end{proof}

	Let $ f : \Gamma \lonto \Delta $ be a surjective $ o $-group homomorphism of divisible $ o $-groups. Then  subgroups $ \Delta' \subseteq \Gamma $ satisfying any of the equivalent conditions (i) - (iii) in Theorem \ref{the.equiv_sect_DOAGS} exist. Indeed, it is well-known that condition (i) in Theorem \ref{the.equiv_sect_DOAGS} is equivalent to 
	\begin{enumerate}[(a)]
	\item $ \Gamma = \Delta' \oplus \text{ker}(f) $ as groups.
		 \item $ \Delta' $ is  a subgroup maximal for subset inclusion in $ \Gamma $ with  $ \Delta' \cap \text{ker}(f) = (0) $.
\end{enumerate}
The equivalence  (i) $ \Leftrightarrow  $ (a) and the implication (a) $ \Ra $ (b) are straightforward. The implication (b) $ \Ra $  (a) follows from the fact that $ \text{ker}(f) $ is a divisible subgroup of $ \Gamma $  (since $ \text{ker}(f) $ is convex in $ \Gamma $) and because divisible groups of a group are absolute direct summands, see  \cite[Theorem 21.2]{fuchs.inf_ab_gps_vol_1}. Applying Zorn's lemma ensures existence of subgroups $ \Delta' \subseteq \Gamma $ satisfying (b). In particular, subgroups $ \Delta'  \subseteq \Gamma$   maximal for subset inclusion in $ \Gamma $ with  $ \Delta' \cap \text{ker}(f) = (0) $ are tame and cofinal in $ \Gamma $.

\section{An Application to Real Closed Valued Fields}\label{sec.rcvf}

Recall that a \textit{real closed valued field}  is a valued field $ (K, v) $ (see \cite{engler/prestel.valued_fields} or \cite[Section 3]{aschenbrenner/vdD/vdHoeven.asympt}) such that $ K $ is a real closed field and $ v $ is a  \textit{convex valuation} (also called \textit{order-compatible valuation}) on $ K $, that is, $ 0 < a < b $ implies  $ v(b) \leq v(a) $ for all $ a, b \in K $. Equivalently, a real closed valued field is a pair $ (K, V) $ where $ K $ is a real closed field and $ V $ is a convex subring. If $ (K, v)  $ is a real closed valued field, then  write $ V_v := \{ a \in K \mid v (a) \geq 0\} $ for its corresponding valuation ring, and regard $ K^{>0} $ as a multiplicative group.

\begin{corollary}\label{cor.equiv_sect_val_grp_rcvf}
	Let $ (K, v) $ be a real closed valued field with value group $ \Gamma $ and $G \subseteq K^{>0} $ be a subgroup. The following are equivalent:
	\begin{enumerate}[\normalfont(i)]
		\item $ G $ is a \emph{monomial group of} $ (K, v) $, that is, the map $ v_{\upharpoonright G} : G \lra \Gamma $ is a group isomorphism. In particular,   $ (v_{\upharpoonright G})^{-1}: \Gamma \linto K^{>0}$ is a cross-section of the group homomorphism $v_{\upharpoonright K^{>0}}: K^{>0} \lonto \Gamma$.
		\item $ G$ is divisible, tame and cofinal in $ K^{>0} $, and $ \emph{ker}(v_{\upharpoonright K^{>0}}) = \{ r \in K^{>0}\mid \emph{st}(r) = 1 \}  $.

		\item $ G$ is divisible, tame and cofinal in $ K^{>0} $, and $ v(r) \geq 0 $ if and only if $ \emph{st}(r) \leq 1 $ for all $ r \in K^{>0} $.
		\item $ G$  is divisible, tame and cofinal in $ K^{>0} $, and $ V_v = \{a \in K \mid a =0 \text{ or } \emph{st}(|a|)\leq 1 \} $.

	\end{enumerate}
	In particular, if any of the items \emph{(i)} - \emph{(iv)} hold, then:
\begin{enumerate}[\normalfont-]
\item 	$ G $ is divisible, tame and cofinal in $ K^{>0}$,
\item $ \emph{st}(r) $ is the unique element in  $ G $ such that  $ v(\emph{st}(r)) =v(r) $ for all $ r \in K^{>0} $, and
\item the standard part map $ \emph{st}: K^{>0}\lonto G $  is a retract of the $ o $-group extension $ G \subseteq K^{>0}$.
\end{enumerate}
\end{corollary}

\begin{proof} 
	Let $ \Gamma^{\text{op}} $ be the $ o $-group obtained by reversing the order of $ \Gamma $, so that the identity map $ \text{id}_{\Gamma}:\Gamma \lra \Gamma^{\text{op}} $ is an order-reversing group isomorphism. Since $ (K,v) $ is a real closed valued field, $ K^{>0} $ and $ \Gamma $ are divisible  $ o $-groups and the composite map $\text{id}_{\Gamma}\circ v_{\upharpoonright K^{>0}} : K^{>0} \onto \Gamma \ra \Gamma^{\text{op}} $ is a surjective $ o $-group homomorphism such that $ \text{ker}(\text{id}_{\Gamma} \circ v_{\upharpoonright K^{>0}}) = \text{ker}(v_{\upharpoonright K^{>0}} )  $, and thus the equivalence of items (i) - (iii)  follows from Theorem \ref{the.equiv_sect_DOAGS} applied to $ \text{id}_{\Gamma} \circ v_{\upharpoonright K^{>0}} $. Moreover,  the equivalence of items (iii) and  (iv) is clear since $ V_v = \{a \in K \mid v(a) \geq 0\} $ and $ v(a) = v(-a) $ for all $ a \in K $.
\end{proof}

\begin{example}\label{ex.sect_hahn_rcvf}
	Let  $ K $ be a real closed field and $ \Gamma $ be a divisible  $ o $-group. Then the field of Hahn series $ K((\Gamma)):= K((x^{\Gamma})) $ is a real closed valued field with value group $ \Gamma $ and  $ x^{\Gamma} := \{x^{\gamma}\mid \gamma \in \Gamma\} $ is a divisible subgroup of $ K((\Gamma))^{>0} $ such that $ v_{\upharpoonright x^{\Gamma}}: x^{\Gamma} \lra \Gamma $ is a  group isomorphism, therefore $ x^{\Gamma} $ is tame in $ K((\Gamma))^{>0} $ and $ \text{st}(r) = x^{v(r)} $ for all $ r \in K((\Gamma))^{>0} $ by Corollary \ref{cor.equiv_sect_val_grp_rcvf}.
\end{example}

Given a real closed valued field $ (K, v) $, one can therefore identify  the monomial groups of $  K$ with certain tame and cofinal divisible subgroups of $ K^{>0}$ by Corollary \ref{cor.equiv_sect_val_grp_rcvf}. Conversely, order-compatible valuations on  $ K $ are induced by certain tame and cofinal divisible subgroups of $ K^{>0} $:

\begin{lemma}\label{lem.tame_subgroup_generates_valution}
	Let $ K $ be a real closed field and $ G \subseteq  K^{>0}$ be divisible subgroup which is tame and cofinal in $ K^{>0} $. Let also $ K^{\times} := K \setminus \{ 0 \}  $ and $ G^{\emph{op}} $ be the $ o $-group obtained by reversing the order of $ G $. The following are equivalent:
\begin{enumerate}[\normalfont(i)]
	\item 	The map $ v_G : K^{\times} \lonto G^{\emph{op}}  $ given by  $ v_G(a) := \emph{st}(|a|)$ is an order-compatible valuation on $ K $. 	\item $ \emph{st}(|a+b|) \leq \max \{ \emph{st}(|a|), \emph{st}(|b|) \}    $ in $ G $ for all $ a, b \in K^{\times} $ with $ a \neq -b $.
\item $ \emph{st}(2) \leq 1 $.
\end{enumerate}
In particular, if any of the conditions \emph{(i)} - \emph{(iii)} hold, then $ G $ is a monomial group of the real closed valued field $ (K, v_G) $ and the corresponding convex valuation ring is $ V_{G}:= \{ 0 \}  \cup \{ a \in K^{\times}\mid \emph{st}(|a|) \leq 1\}   $.

\end{lemma}

\begin{proof}
	\underline{(i) $ \Ra $  (ii).} Clear since $ v_G(a+b) \geq \min\{ v_G(a), v_G(b) \}  $ in $ G^{\text{op}} $ if and only if $ \text{st}(|a+b|) \leq \max \{ \text{st}(|a|), \text{st}(|b|) \}    $ in $ G $ for all $ a,b \in K^{\times} $ for all $ a \neq -b $.

	\underline{(ii) $ \Ra $  (iii).} Clear from the fact that $ \text{st}(1)= 1 $.

	\underline{(iii) $ \Ra $  (i).} 	By choice of $ K $ and $ G $, it follows from  Proposition \ref{prop.retract_tame_subgrp} that the standard part map $ \text{st} : K^{>0} \lonto G  $ is a surjective morphism of  divisible $ o $-groups, and thus $ v_G: K^{\times} \lra G^{\text{op}} $ is a surjective group homomorphism  such that $  a < b  $ in $ K^{>0} $ implies $ v_G(b) \leq v_G(a) $ in $ G^{\text{op}} $. It remains to show that for all $ a , b \in K^{\times} $ with $ a \neq -b $, $ v_G(a+b) \geq \min\{ v_G(a), v_G(b) \}  $ in $ G^{\text{op}} $, i.e., $ \text{st}(|a+b|) \leq \max \{ \text{st}(|a|), \text{st}(|b|) \}    $ in $ G $. Pick $ a, b \in K^{\times} $ with $ a \neq -b $ and assume without loss of generality that $ |a | \leq |b| $, so that $ \max \{ \text{st}(|a|), \text{st}(|b|) \}  = \text{st}(|b|) $; then $ |a+b| \leq |a|+ |b| \leq 2|b| $, therefore $ \text{st}(|a+b|)\leq \text{st}(2)\text{st}(|b|) \leq \text{st}(|b|)$, as required.

	The last statement follows by the equivalence of (i) $ \Leftrightarrow $ (iv) in Corollary \ref{cor.equiv_sect_val_grp_rcvf}.
\end{proof}

Let $ \mathscr{L}^{\text{poring}}:= \{ +, -, \cdot,  0, 1, \leq \}  $ be the language of partially ordered rings and $ \mathscr{L}_P:= \mathscr{L}^{\text{poring}} \cup \{ P \} $ where $ P $ is a unary relation symbol. Define $ T_P $ to be the $ \mathscr{L}_P $-extension of the theory of real closed fields stating that $ P $ is a divisible, tame and cofinal subgroup of  the positive multiplicative units of the field  which satisfies any of the items (i) - (iii) in Lemma \ref{lem.tame_subgroup_generates_valution}. Then  $ T_P $ is a one-sorted axiomatization of the class of real closed valued fields in which the value group is definable (as well as the corresponding convex valuation ring). One may also enrich the language $ \mathscr{L}_P $ with a new unary predicate $ \mathscr{L}_{P,Q}:= \mathscr{L}_{P} \cup \{ Q \}  $, and consider the $ \mathscr{L}_{P, Q} $-extension of $ T_{P} $ stating that $ Q $ is a tame real closed subfield of the underlying field $ K $ such that the convex hull of $ Q $ in $ K $ is equal to the convex valuation ring $ V $ induced by $ P $ in the sense of Lemma \ref{lem.tame_subgroup_generates_valution}. If $ (K; G, L) \models T_{P, Q} $, then by Remark 2.11 and Theorem 2.12 in \cite{vdD/Lewenberg.T-convexity_and_tame_extensions} it follows that $ L $ is isomorphic to the residue field of $ (K; G) \models T_{P} $, therefore $ T_{P, Q} $ is a one-sorted axiomatization of the class of real closed valued fields in which both the value group and the residue field are definable. 

\printbibliography

\typeout{get arXiv to do 4 passes: Label(s) may have changed. Rerun}

\end{document}